\documentclass[twoside,
]{amsart}

\usepackage{srcltx}
  \usepackage[utf8]{inputenc}
  \usepackage{amsmath}
  \usepackage{amssymb}
  \usepackage{amsthm}
  \usepackage{mathrsfs}
  \usepackage{stmaryrd}
  \usepackage{bbm}
  \usepackage{pifont}
  \usepackage{graphicx}
  \usepackage{xmpmulti}
  \usepackage{color}

	\newtheoremstyle{slanted}
	{}
	{}
	{\slshape}
	{}
	{\bfseries}
	{.}
	{ }
	{}
	
	\theoremstyle{slanted}
	\newtheorem{theo}{Theorem}[section]
	\newtheorem{prop}[theo]{Proposition}

	\newtheorem{lemma}[theo]{Lemma}
	
	\newtheorem{corollary}[theo]{Corollary}
	
	\newtheorem{remark}[theo]{Remark}

	\def\egdef{:=}

	\newcommand{\tend}[3][]{\xrightarrow[#2\to#3]{#1}}

	\newcommand{\ZZ}{\mathbb{Z}}

	\newcommand{\RR}{\mathbb{R}}

	\renewcommand{\O}{\mathcal{O}}

	\newcommand{\brg}{\llbracket}
	\newcommand{\brd}{\rrbracket}
	\newcommand{\MAX}{\mbox{MAX}}
	\newcommand{\SUCC}{\mbox{succ}}

\title{Dynamics of $\lambda$-continued fractions and $\beta$-shifts}

\author{Élise Janvresse, Benoît Rittaud  and Thierry de la Rue}

\address{\'Elise Janvresse, Thierry de la Rue:
Laboratoire de Math\'ematiques Rapha\"el Salem, 
Universit\'e de Rouen, CNRS -- 
Avenue de l'Universit\'e -- 
F76801 Saint \'Etienne du Rouvray.}
\email{Elise.Janvresse@univ-rouen.fr\\Thierry.de-la-Rue@univ-rouen.fr}
\address{Beno\^it Rittaud: Laboratoire Analyse, G\'eom\'etrie et Applications, Universit\'e Paris~13 Institut Galil\'ee, CNRS -- 
99 avenue Jean-Baptiste Cl\'ement -- 
F93430 Villetaneuse.}
\email{rittaud@math.univ-paris13.fr}

\begin{document}
\bibliographystyle{amsplain}

\begin{abstract}
For a real number $0<\lambda<2$, we introduce a transformation $T_\lambda$ naturally associated to expansion in $\lambda$-continued fraction, for which we also give a geometrical interpretation. 
The symbolic coding of the orbits of $T_\lambda$ provides an algorithm to expand any positive real number in $\lambda$-continued fraction. We prove the conjugacy between $T_\lambda$ and some $\beta$-shift, $\beta>1$. Some properties of the map $\lambda\mapsto\beta(\lambda)$ are established: It is increasing and continuous from $]0, 2[$ onto $]1,\infty[$ but non-analytic. 
\end{abstract}

\maketitle

\section{Introduction}

In all the paper, $\lambda$ denotes a real number, $0<\lambda<2$. A \emph{$\lambda$-continued fraction} is an expression of the form
$$ [a_0,\ldots, a_n,\ldots]_\lambda\egdef a_0\lambda + \cfrac{1}{a_1\lambda + \cfrac{1}{\ddots+\cfrac{1}{a_n \lambda +_{\ddots}}}} $$
where $(a_n)_{n\ge0}$ is a finite or infinite sequence, with $a_n\in\ZZ\setminus\{0\}$ for $n\ge1$.
This kind of continued fractions has been studied by Rosen in~\cite{rosen1954}, where specific properties are enlightened when $\lambda=\lambda_k \egdef 2\cos (\pi/k)$ for some integer $k\ge 3$.
For $0< \lambda <2$, any real number can be expanded in $\lambda$-continued fraction, even if the expansion is not unique in general. In this paper we study a transformation associated to a particular expansion in $\lambda$-continued fractions, in which we always have $a_0=0$, $a_1>0$ and the signs of the $a_n$'s alternate.

The motivations for the present article stem from several works by the same authors~\cite{janvresse2008,janvresse2010}, where the exponential growth of random Fibonacci sequences with parameter $\lambda$ is studied. 
In~\cite{janvresse2010}, the case $\lambda=\lambda_k$ for some integer $k\ge 3$ is solved and involves a probability distribution on $\RR_+$ invariant under some dynamics. 
This measure is defined inductively on generalized Stern-Brocot intervals, whose endpoints are described in terms of finite expansion in $\lambda$-continued fraction.
A key fact proved in~\cite{janvresse2010} is that the sequence of partitions of $\RR_+$ given by generalized Stern-Brocot intervals for $\lambda=\lambda_k$ is isomorphic to the sequence of partitions of $[0,1]$ associated to the expansion of real numbers in base $(k-1)$. 
In this paper we investigate the link between $\lambda$-continued fractions and expansions in non-integer basis, generalizing the correspondence observed in the case $\lambda=\lambda_k$. 

\subsection*{Roadmap}
We introduce in Section~\ref{Sec:transformation} a transformation $T_\lambda$ naturally associated to expansion in $\lambda$-continued fraction, for which we also give a nice geometrical interpretation. The study of this transformation leads to a symbolic coding of the orbits. In Section~\ref{Sec:distinct codings}, we prove that distinct points have different codings, which provides an algorithm to expand any positive real number in $\lambda$-continued fraction (see Section~\ref{Sec:continued fractions}). 
In Section~\ref{Sec:conjugacy} we give a characterization of symbolic coding of orbits for $T_\lambda$ (Theorem~\ref{Thm:characterization}), which enables to prove the conjugacy between $T_\lambda$ and some $\beta$-shift, $\beta>1$ (Theorem~\ref{Thm:correspondence}). In Section~\ref{Sec:map}, we present some properties of the map $\lambda\mapsto\beta(\lambda)$: It is increasing and continuous from $]0, 2[$ onto $]1,\infty[$ but non-analytic (Theorem~\ref{Thm:beta(lambda)} and Corollary~\ref{Cor:non-analytic}). 
We end up by raising some open questions on the subject.

\section{Description of the dynamics}
\label{Sec:transformation}
\subsection{The transformation $T_\lambda$}

We start by defining the homographic functions $h$ and $h_0$ by setting
$$ h(y) \egdef \dfrac{1}{\lambda-y}\ ;\ h_0(y)\egdef \dfrac{y}{\lambda y+1}. $$ 
Observe that, when $y$ ranges over $[0,\infty[$, $h_0(y)$ increases from $m_0^\lambda\egdef 0$ to $m_1^\lambda\egdef 1/\lambda=h(m_0^\lambda)$. 
We recursively define the sequence $(m_i^\lambda)$ by setting, while $m_i^\lambda<\lambda$, $m_{i+1}^\lambda\egdef h(m_i^\lambda)$. 

\begin{lemma}
\label{Lemma:definition des m_i}
The sequence $(m_i^\lambda)$ is increasing, and there exists $i_\lambda\ge1$ such that $m_{i_\lambda}^\lambda\ge\lambda$. 
\end{lemma}

\begin{proof}
 Since $h$ is increasing on $[0,\lambda[$ and $m_0^\lambda<m_1^\lambda$, an easy induction shows that $(m_i^\lambda)$ is increasing. If the sequence $(m_i^\lambda)$ were infinite, it would be bounded above by $\lambda$, therefore it would converge to a fixed point for $h$. But, as $\lambda<2$, such a fixed point does not exist.
\end{proof}
The sequence $(m_i^\lambda)$ is thus finite, with $(i_\lambda+1)$ terms satisfying 
$$ 0=m_0^\lambda<m_1^\lambda<\cdots<m_{i_\lambda-1}^\lambda<\lambda\le m_{i_\lambda}^\lambda <\infty. $$

We now recursively define the homographic functions $(h_i)_{0\le i \le i_\lambda}$ by 
$$h_{i+1}(y) \egdef h\circ h_i(y). $$
Note that, for $i<i_\lambda$, $I_i^\lambda\egdef h_i([0,\infty[)=[m_i^\lambda,m_{i+1}^\lambda[$, and that the function $h_i$ has no pole on $[0,\infty[$. 
If $m_{i_\lambda}^\lambda>\lambda$, the last function $h_{i_\lambda}$ has a pole $\ell_\lambda\egdef h_{i_\lambda-1}^{-1}(\lambda)$ and $h_{i_\lambda}([0,\ell_{\lambda}[)=[m_{i_\lambda}^\lambda,\infty[$. If $m_{i_\lambda}^\lambda=\lambda$, $h_{i_\lambda}$ maps $[0,\infty[$ onto $[m_{i_\lambda}^\lambda,\infty[$ (and is thus a degenerate homographic function: $h_{i_\lambda}$ is affine in this case). 
We will consider $h_{i_\lambda}$ to be defined only on $[0,\ell_{\lambda}[$, where $\ell_{\lambda}=\infty$ in the second case, so that $I_{i_\lambda}^\lambda\egdef h_{i_\lambda}([0,\ell_{\lambda}[)=[m_{i_\lambda}^\lambda,\infty[$.

{From} the above, it follows that $\{I_i^\lambda,\ 0\le i\le i_\lambda\}$ is a partition of $[0,\infty[$. 
We then define the transformation $T_\lambda:\ [0,\infty[\longrightarrow [0,\infty[$ by setting 
$$ \forall x\in I_i^\lambda,\ T_\lambda(x)\egdef h_i^{-1}(x). $$
(See Figure~\ref{Fig:h_i}.)

\begin{figure}
\input{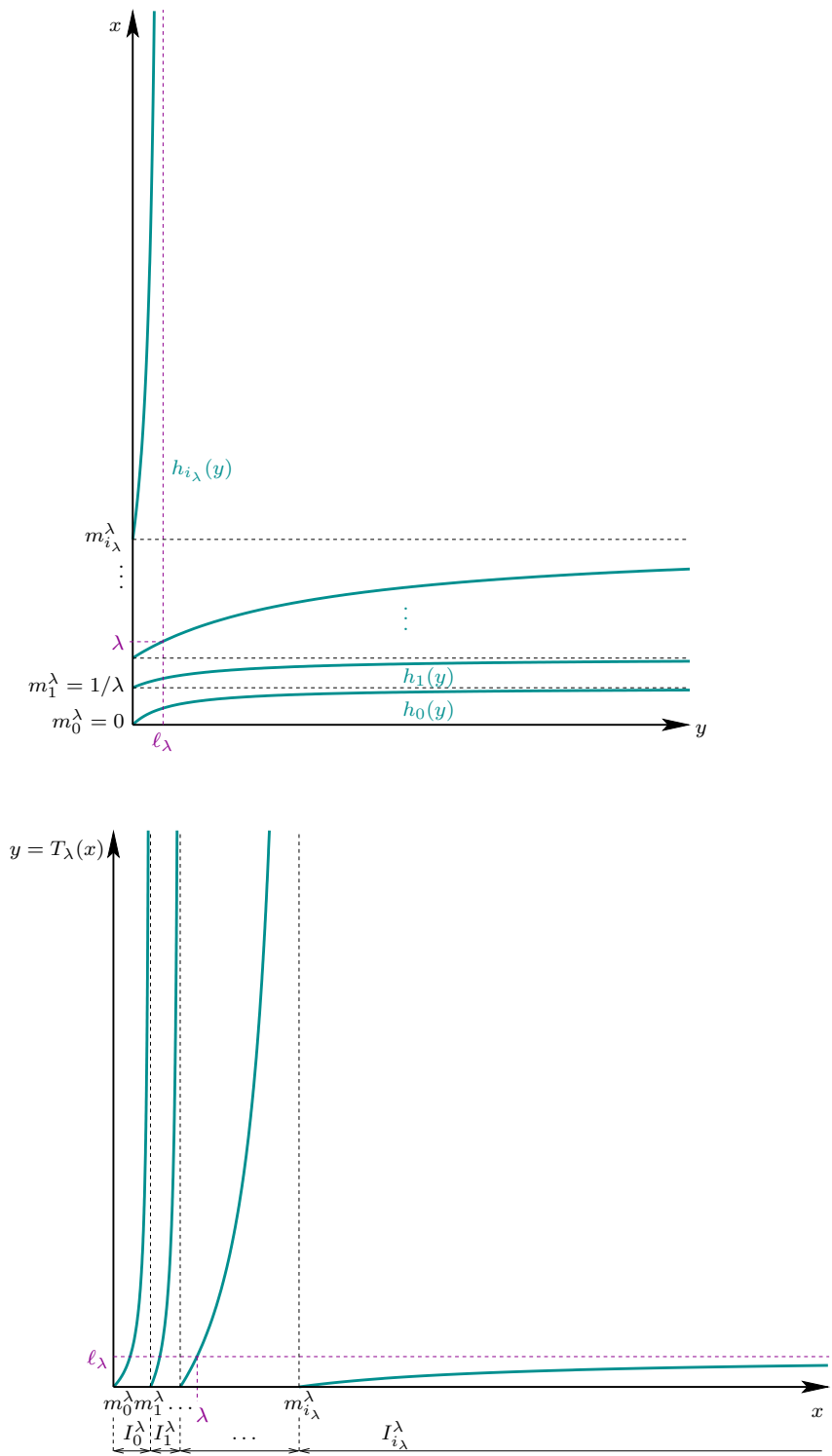}
\caption{Top: Graphs of the homographic functions $h_i$.
Bottom: Graph of the transformation $T_\lambda$. (Here $\lambda=1.5$)} 
\label{Fig:h_i}
\end{figure}

\subsection{Coding of the orbits}

The way $T_\lambda$ is defined naturally leads to a symbolic coding of the orbits under the action of $T_\lambda$, on the alphabet $\{0, \ldots, i_\lambda\}$: For $x\in [0,\infty[$, we set 
$$ \omega_\lambda(x) \egdef x_0 x_1 x_2 \ldots, $$
where $x_\ell$ is the unique element of $\{0, \ldots, i_\lambda\}$ such that $T_\lambda^\ell(x)\in I_{x_\ell}^\lambda$. 

It will be useful in the sequel to consider the lexicographic order on $\{0, \ldots, i_\lambda\}^{\ZZ_+}$: 
If $\omega=(\omega_j)$ and $\omega'=(\omega_j')$ belong to $\{0, \ldots, i_\lambda\}^{\ZZ_+}$, we shall note $\omega\prec\omega'$ or $\omega'\succ\omega$ whenever there exists $\ell\ge0$ such that $\omega_\ell<\omega'_\ell$ and $\omega_j=\omega'_j$ for each $0\le j<\ell$. By $\omega\preceq\omega'$ (or $\omega'\succeq\omega$), we mean $\omega\prec\omega'$ or $\omega=\omega'$.

\begin{lemma}\label{lemma:omega croissant}
 If $0\le x<x'<\infty$, $\omega_\lambda(x) \preceq\omega_\lambda(x')$.
\end{lemma}

\begin{proof}
If $\omega_\lambda(x) \neq \omega_\lambda(x')$, let $\ell$ be the smallest integer such that $x_\ell\neq x'_\ell$. Since the functions $h_i^{-1}$ are increasing, it is easily proved by induction on $j$ that for $0\le j\le \ell$, $T_\lambda^j (x) < T_\lambda^j (x')$. Hence $T_\lambda^\ell (x) < T_\lambda^\ell (x')$ and $x_\ell<x'_\ell$.
\end{proof}

In Section~\ref{Sec:distinct codings}, we prove that in fact $\omega_\lambda(x) \neq\omega_\lambda(x')$ if $x\neq x'$.

\medskip

An object of particular interest in the study of $T_\lambda$ will be the limit, as $x\to\infty$, of $\omega_\lambda(x)$ (which exists by Lemma~\ref{lemma:omega croissant}). We shall denote this limit
$$ \omega_\lambda(\infty) = \infty_0 \infty_1\infty_2\ldots \egdef \lim_{x\to\infty}\uparrow\omega_\lambda(x). $$

For $n\ge0$ and $a_0,\ldots,a_n\in\{0, \ldots, i_\lambda\}$, we define
$$ I_{a_0\ldots a_n}^\lambda \egdef \{x:\ x_j=a_j,\ 0\le j\le n \}. $$

It is easily checked by induction that $I_{a_0\ldots a_n}^\lambda$ is either empty or an interval of the form $h_{a_0}\circ\cdots\circ h_{a_n}([0,r_{a_0\ldots a_n}[)$ where $0<r_{a_0\ldots a_n}\le\infty$.

\medskip

We denote by $\sigma$ the shift on $\{0, \ldots, i_\lambda\}^{\ZZ_+}$, so that $\omega_\lambda(T_\lambda (x)) = \sigma \omega_\lambda(x)$. 

\begin{lemma}
\label{lemma:lexico}
 For each $0\le x\le\infty$, and each $\ell\ge0$, $\sigma^\ell \omega_\lambda(x)\preceq\omega_\lambda(\infty)$.
\end{lemma}
\begin{proof}
This is an immediate consequence of the definition of $\omega_\lambda(\infty)$ and Lemma~\ref{lemma:omega croissant}.
\end{proof}

\begin{remark}
\label{Rmk:orbites}
There exist some links between $\omega_\lambda(\infty)$, $\omega_\lambda(\lambda)$ and $\omega_\lambda(\ell_\lambda)$. 
Let $\omega_\lambda(\lambda)=a_0a_1a_2\dots$ If $T_\lambda(\lambda)\not=0$, then $\omega_\lambda(\ell_\lambda)=\sigma(\omega_\lambda(\lambda))=a_1a_2\dots$
If $T_\lambda^j(\lambda)\not=0$ for any $j\ge1$, then 
$\omega_\lambda(\infty)=(a_0+1)\omega_\lambda(\ell_\lambda)=(a_0+1)a_1a_2\dots$
\end{remark}

\subsection{Matrices associated to the homographic functions $h_i$}
Recall that each homographic function can be written in the form $x\mapsto\dfrac{ax+b}{cx+d}$ where $ad-bc=1$ and is associated to the matrix 
$\begin{pmatrix}
a & b\\
c & d
\end{pmatrix}$. Composition of homographic functions corresponds to matrix multiplication. 

We thus introduce the matrices 
$$
H\egdef
\begin{pmatrix}
0 & 1\\
-1 & \lambda
\end{pmatrix}
\quad \text{and}\quad
H_0\egdef
\begin{pmatrix}
1 & 0\\
\lambda & 1
\end{pmatrix}
$$
respectively associated to $h$ and $h_0$. For $0\le i\le i_\lambda$, let $H_i\egdef H^iH_0$ be the matrix associated to $h_i$. An easy induction shows that $H_i$ is of the form
$$
H_i=
\begin{pmatrix}
P_{i+1}(\lambda) & P_{i}(\lambda)\\
P_{i+2}(\lambda) & P_{i+1}(\lambda)
\end{pmatrix},
$$
where the sequence of polynomials $(P_i)$ is defined by
$$ P_0(X)\egdef 0,\quad P_1(X)\egdef 1,\quad\text{and}\quad P_{i+2}(X)\egdef X P_{i+1}(X)-P_i(X). $$

\begin{lemma}
\label{Lemma:signOfPi}
 For $1\le i\le i_\lambda+1$, $P_i(\lambda)>0$, and $P_{i_\lambda+2}(\lambda)\le0$.
\end{lemma}

\begin{proof}
Since, for $i<i_\lambda$, the function $h_i$ has no pole on $[0,\infty[$, all the $P_i(\lambda)$ are nonnegative for $1\le i\le i_\lambda+1$. It follows that, for $1\le i\le i_\lambda$, $P_i(\lambda)>0$ (for, if $P_i(\lambda)=0$ and $P_{i-1}(\lambda)>0$, then $P_{i+1}(\lambda)<0$). If we had $P_{i_\lambda+1}(\lambda)=0$, $h_{i_\lambda-1}$ would be an affine function with positive slope, and then $m_{i_\lambda}^\lambda$ would be $\infty$, which is not possible. Hence $P_{i_\lambda+1}>0$.

Finally, since $[m_{i_\lambda}^\lambda,\infty[=h_{i_\lambda}([0,\ell_\lambda[)$, we have $P_{i_\lambda+2}\le 0$ (otherwise $h_{i_\lambda}$ would be bounded on $[0,\ell_\lambda[$).
\end{proof}

\subsection{Geometrical interpretation of $P_i(\lambda)$}
\label{Section:Geometrical Interpretation of P_i}

The key observation is the following: Let $\theta\in ]0,\pi/2[$ be such that $\lambda=2\cos\theta$. 
Fix two points $M,M'$ on a circle centered at the origin $O$, such that the oriented angle $(OM,OM')$ equals $\theta$. 
Let $M''$ be the image of $M'$ by the rotation of angle $\theta$ and center $O$. Then the respective abscissae $t$, $t'$ and $t''$ of $M$, $M'$ and $M''$ satisfy $t''=\lambda t' - t$. 

Let us consider the circle centered at the origin with radius $R=1/\cos(\theta-\pi/2)$.
We fix on the circle the point $M_0=(0,-R)$, and define the sequence of points $(M_i)$ such that $M_i$ is the image of $M_{i-1}$ by the rotation of angle $\theta$ and center $O$. 
Let $t_i$ be the abscissa of $M_i$. 
Observe that $t_0=0=P_0(\lambda)$, $t_1=1=P_1(\lambda)$ by choice of $R$, and by induction $t_i=P_i(\lambda)$ for all $i\ge0$. (See Figure~\ref{Fig:cercle}.)

\begin{figure}[h]
\input{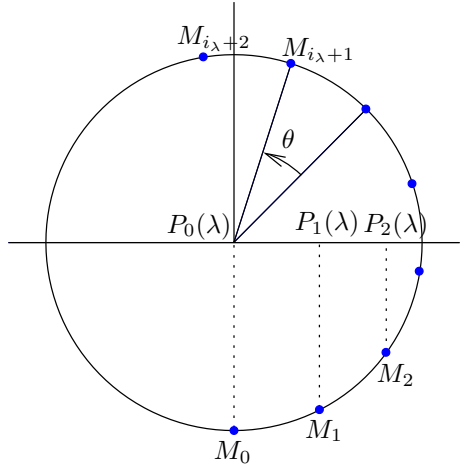}
\caption{Geometrical interpretation of the sequence $P_i(\lambda)$ as the successive abscissae of points on the circle centered at the origin with radius $R=1/\cos(\pi/2-\theta)$, where $\theta$ is such that $\lambda=2\cos\theta$.}
\label{Fig:cercle}
\end{figure} 

\begin{prop}
\label{Prop:i_lambda}
 Let us define the increasing sequence $(\lambda_k)_{k\ge2}$ by $\lambda_k\egdef2\cos(\pi/k)$. 
Then $i_\lambda=k-2$ for $\lambda\in]\lambda_{k-1},\lambda_{k}]$ ($\forall k\ge3$).
\end{prop}

\begin{proof}
Recall that $i_\lambda$ is characterized by the fact that $P_{i_\lambda+2}(\lambda)\le 0$ and $P_i(\lambda)>0$ for $1\le i\le i_\lambda+1$. 
Since for $\lambda=2\cos \theta \in]\lambda_{k-1},\lambda_{k}]$, we have $\pi/k\le \theta<\pi/(k-1)$, the result of Proposition~\ref{Prop:i_lambda} is a direct consequence of our geometrical interpretation. 
%
\end{proof}

Besides, we see that the $P_i(\lambda)$'s are bounded by $R=1/\cos(\pi/2-\theta)$, and that
\begin{equation}
\label{eq:P_j inequalities}
 P_i(\lambda)>1\quad\text{for }2\le i\le i_\lambda,\quad\text{and }0\le P_{i_\lambda+1}(\lambda)\le1.
\end{equation}

%
%

\subsection{Geometrical interpretation of $T_\lambda$}
\label{Section:Geometrical Interpretation}
By definition, $T_\lambda(x)$ can be obtained by the following recursive algorithm:

If $0\le x<1/\lambda$, $T_\lambda(x)=h_0^{-1}(x)=\dfrac{x}{1-\lambda x}$, 

else $T_\lambda(x)=T_\lambda(h^{-1}(x))=T_\lambda\left( \dfrac{\lambda x-1}{x}\right)$.

Suppose that $x$ is written $x=t_1/t_0$, with $t_1\ge0$, $t_0>0$. By Lemma~6.4 in~\cite{janvresse2010}, we can find a circle centered at the origin and two points $M_0$ and $M_1$ on this circle with respective abscissae $t_0$ and $t_1$, such that the oriented angle $(OM_0,OM_1)$ equals~$\theta$. Let $M_2$ be the image of $M_1$ by the rotation of angle $\theta$, and denote its abscissa by $t_2=\lambda t_1-t_0$. 

If $x<1/\lambda$, then $t_2<0$. We get 
$$T_\lambda(x)=h_0^{-1}(x)=\dfrac{t_1}{t_0-\lambda t_1}=\dfrac{t_1}{-t_2} ,
\mbox{ and }
H_0 \begin{pmatrix}   t_1  \\ -t_2  \end{pmatrix}
= \begin{pmatrix}  t_1  \\ t_0  \end{pmatrix}.
$$
Else, we have $t_2\ge0$. We obtain 
$$
T_\lambda(x)=T_\lambda(h^{-1}(x))=T_\lambda\left( \dfrac{\lambda t_1-t_0}{t_1}\right)=T_\lambda\left(\dfrac{t_2}{t_1}\right)
\mbox{ and }
H \begin{pmatrix}  t_2  \\ t_1  \end{pmatrix}
=  \begin{pmatrix}  t_1  \\ t_0  \end{pmatrix}.
$$
We recursively define the points $M_j$ on the circle, where $M_j$ is the image of $M_{j-1}$ by the rotation of angle $\theta$. Denote by $t_j$ the abscissa of $M_j$ and let $i(x)$ be such that $t_{i(x)+1}$ is the first negative abscissa. 
Then $x\in I_{i(x)-1}^\lambda$, and
\begin{equation}
 \label{eq:T_lambda}
T_\lambda(x)=\dfrac{t_{i(x)}}{-t_{i(x)+1}}=T_\lambda\left(\dfrac{t_{j+1}}{t_j}\right)\quad \forall 0\le j<i(x)\ .
\end{equation}
Moreover, 
\begin{equation}
 \label{eq:T_lambda matrices}
H_{i(x)} \begin{pmatrix}  t_{i(x)}  \\ -t_{i(x)+1}  \end{pmatrix}
=  \begin{pmatrix}  t_1  \\ t_0  \end{pmatrix}.
\end{equation}

If we want to iterate $T_\lambda$ in this setting, we have to find a new circle centered at the origin and two points $N_0$ and $N_1$ on this circle with respective abscissae $-t_{i(x)+1}$ and $t_{i(x)}$, such that the oriented angle $(ON_0,ON_1)$ equals~$\theta$. Let us denote by $R'$ the radius of the new circle, whereas $R$ stands for the radius of the first circle. 
(See Figure~\ref{Fig:cercles}.)

\begin{prop}
\label{Prop:R decroit}
With the above notations, we always have $R'\le R$. Moreover, there exists $K(\lambda)>1$ such that $R'\le R/K(\lambda)$ whenever $T_\lambda(x)$ and $T_\lambda^2(x)$ do not both belong to $I_0^\lambda$, and $T_\lambda(x)\notin I_{i_\lambda}^\lambda$.
\end{prop}

\begin{figure}[h]
\input{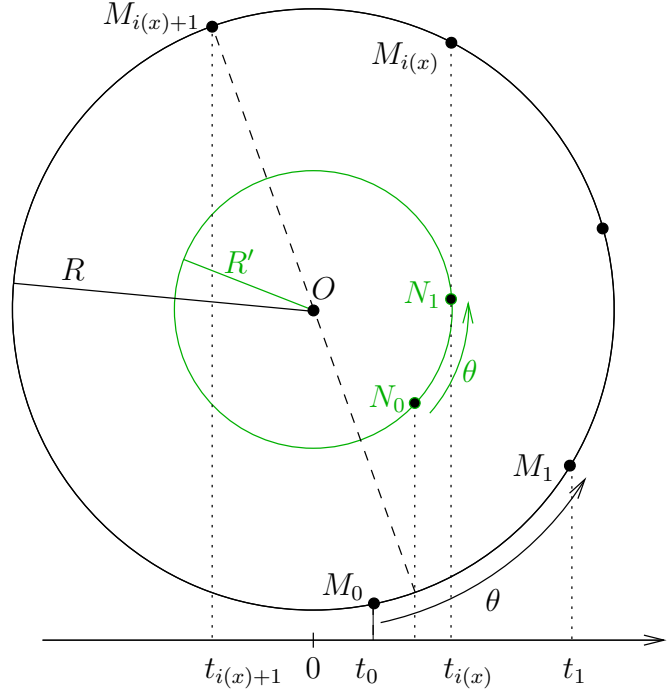}
\caption{Geometrical interpretation of $T_\lambda$: The transformation maps $x=t_1/t_0$ to $t_{i(x)}/(-t_{i(x)+1})$. The figure also illustrates the evolution of the radius of the circle.}
\label{Fig:cercles}
\end{figure} 

\begin{proof}
Let $\alpha$ be the argument of $M_{i(x)}$, so that $t_{i(x)}=R\cos \alpha$ and $t_{i(x)+1}=R\cos(\alpha+\theta)$. Observe that $$ \frac{\pi}{2}-\theta < \alpha \le \frac{\pi}{2} $$
because $t_{i(x)+1}<0$ and $t_{i(x)}\ge 0$. Let $\tau$ be the argument of $N_1$, so that $t_{i(x)}=R'\cos\tau$ and $-t_{i(x)+1}=R'\cos(\tau-\theta)$. Observe that, since $-t_{i(x)+1}>0$, $\tau-\theta>-\pi/2$, hence 
\begin{equation}
 \label{eq:tau}
 \tau>-\alpha.
\end{equation}
We want to estimate the ratio
$$ \frac{R}{R'} = \frac{\cos\tau}{\cos\alpha} = \frac{\cos(\tau-\theta)}{\cos(\alpha+\theta+\pi)} . $$
Write the last equality in the form
$$ \frac{\cos(\tau-\theta)}{\cos\tau} = \frac{\cos(\alpha+\theta+\pi)}{\cos\alpha}.$$
Expanding the cosines yields
\begin{equation}
 \label{eq:tan}
 \tan\tau =\tan\alpha - 2\frac{\cos\theta}{\sin\theta}\le\tan\alpha-\lambda, 
\end{equation}
hence $ \tau \le \alpha$ (where equality only holds when $\tan\alpha=\infty$, that is $\alpha=\pi/2$). Together with~\eqref{eq:tau}, this proves $R/R'\ge 1$.

\medskip

We are now going to prove that under the additional hypothesis of the Proposition, there exists $\delta(\lambda)>0$ such that
\begin{equation}
 \label{eq:alpha}
\frac{\pi}{2}-\theta +\delta(\lambda)< \alpha \le \frac{\pi}{2}-\delta(\lambda).
\end{equation}
Indeed, if $T_\lambda(x)$ and $T_\lambda^2(x)$ do not both belong to $I_0^\lambda$, and $T_\lambda(t_{i(x)}/t_{i(x)-1})=T_\lambda(x)$, we deduce that $t_{i(x)}/t_{i(x)-1}\notin I_{000}^\lambda$. 
Therefore 
$$ \frac{t_{i(x)}}{t_{i(x)-1}}=\frac{\cos\alpha}{\cos(\alpha-\theta)} \ge \sup I_{000}^\lambda > 0. $$
The inequality on the righthand side of~\eqref{eq:alpha} follows.

Assuming moreover that $T_\lambda(x)\notin I_{i_\lambda}^\lambda$, we get by~\eqref{eq:T_lambda}
$$ \frac{t_{i(x)}}{-t_{i(x)+1}} = -\frac{\cos\alpha}{\cos(\alpha+\theta)} < \min I_{i_\lambda}^\lambda, $$
which yields the inequality on the lefthand side of~\eqref{eq:alpha}.

\medskip

Recalling that $\tau>\theta-\pi/2$, we get by the lefthand side~\eqref{eq:alpha} that $\tau>-\alpha+\delta(\lambda)$.
Moreover, \eqref{eq:tan} together with the righthand side of~\eqref{eq:alpha} proves that there exists a constant $C(\lambda)>0$ such that $\tau<\alpha-C(\lambda)$.
Therefore, there exists a constant $K(\lambda)>1$ such that $R/R'=\cos\tau /\cos\alpha>K(\lambda)$.
\end{proof}

\section{Distinct points have different codings}
\label{Sec:distinct codings}

The purpose of this section is to prove the following statement:

\begin{theo}
\label{Thm:distinct_points}
For all $0\le x< x'\le\infty$, $\omega_\lambda(x)\neq\omega_\lambda(x')$
\end{theo}

The first step is an elementary particular case.

\begin{lemma}
\label{lemma:zero}
For any $n\ge 0$, for any $a_0, \dots, a_{n-1}$, there is at most one $x$ such that $\omega_\lambda(x)=a_0a_1\ldots a_{n-1} 0\ldots 0 \ldots$.
\end{lemma}
\begin{proof}
We first prove that $\omega_\lambda(x)=0\ldots  0 \ldots \Rightarrow x=0$.
Note that for all $0\le y<\infty$, $h_0(y)\le y$, with equality only when $y=0$. If $\omega_\lambda(x)=0\ldots  0 \ldots$, for each $n$, $x=h_0^n(T_\lambda^nx)$ and the sequence $(T_\lambda^nx)$ is increasing, and bounded by $m_1^\lambda=1/\lambda$. Therefore it converges to a fixed point for $h_0$, that is 0. This is possible only if $T_\lambda^nx=0$ for all $n$.

The result follows by an easy induction on $n$, using the fact that $T_\lambda$ restricted to an interval $I_{i}^\lambda$ is one-to-one.
\end{proof}

As a corollary, we get the following lemma which will be useful.
\begin{lemma}
\label{Lemma:zeros}
 There exist infinitely many $n$ such that $\infty_n\neq0$.
\end{lemma}
\begin{proof}
 If $\omega_\lambda(\infty)=\infty_0\infty_1\ldots \infty_n 0\ldots 0 \ldots$, then by definition of $\omega_\lambda(\infty)$ we would have $\omega_\lambda(x)=\infty_0\infty_1\ldots \infty_n 0\ldots 0 \ldots$ for all large enough $x$. But this would contradict Lemma~\ref{lemma:zero}.
\end{proof}

\subsection{Infinity is unreachable}

The second step of the proof deals with the case of~$\infty$. 
\begin{prop}
\label{Prop:infinityUnreachable}
 For $x\in[0,\infty[$, $\omega_\lambda(x)\prec\omega_\lambda(\infty)$.
\end{prop}

For a fixed $n\ge0$, the intervals $I_{a_0\ldots a_n}^\lambda$ form a partition of $\RR_+$ associated to the coding of $T_\lambda$ up to time $n$.
The object of interest is here the decreasing sequence of rightmost intervals $(I_{\infty_0\ldots\infty_n}^\lambda)_n$ in the successive partitions. Observe that an equivalent statement to the above proposition is: For all $x\in[0,\infty[$, there exists $n$ large enough such that $x\notin I_{\infty_0\ldots\infty_n}^\lambda$.  

Before turning to the proof of Proposition~\ref{Prop:infinityUnreachable}, we prove a few results about matrices associated to iterates of the homographic functions $h_i$.

\subsubsection{Matrices}
\begin{lemma}
\label{Lemma:coeffMatrice}
 For each ${a_0 \ldots a_n}$ such that $I_{a_0 \ldots a_n}^\lambda\neq\emptyset$, the matrix $H_{a_0}\cdots H_{a_n}$ associated to $h_{a_0}\circ\cdots\circ h_{a_n}$ is of the form 
$\begin{pmatrix}
\alpha & \beta\\
\gamma & \delta
\end{pmatrix}$
with $\beta\ge0$ and $\delta>0$. 
Moreover, $\gamma\le 0$ if $a_0 \ldots a_n=\infty_0\ldots\infty_n$.
\end{lemma}
\begin{proof}
We prove the result by induction on $n$. This is true for $n=0$ by Lemma~\ref{Lemma:signOfPi} (observe that $\infty_0=i_\lambda$). 
Assume the result is true up to $n$, and consider 
$\begin{pmatrix}
\alpha & \beta\\
\gamma & \delta
\end{pmatrix}=H_{a_0}\cdots H_{a_{n+1}}=H_{a_0}\Bigl(H_{a_1}\cdots H_{a_{n+1}}\Bigr)$. 
Observe that Lemma~\ref{Lemma:signOfPi} also ensures the positivity of the upper left coefficient of $H_{a_0}$. 
Hence, by induction hypothesis, $\beta\ge0$. Moreover, $\dfrac{\beta}{\delta}=h_{a_0}\circ\cdots\circ h_{a_{n+1}}(0)\ge0$, thus $\delta>0$.

If we have $\gamma>0$, since $\delta>0$ the homographic function associated to the matrix is bounded on $[0,\infty[$. This is impossible for $a_0 \ldots a_n=\infty_0\ldots\infty_n$, as $I_{\infty_0\ldots\infty_n}^\lambda$ is the rightmost interval of the partition at order $n+1$.
\end{proof}

\begin{lemma}
\label{Lemma:deltadecroissant}
Let $H^{(n)} = 
\begin{pmatrix}
\alpha_n & \beta_n\\
\gamma_n & \delta_n
\end{pmatrix}
\egdef H_{\infty_0}\cdots H_{\infty_n}$. Then  $\delta_{n+1}\le \delta_n$ for all $n\ge 0$.
\end{lemma}

\begin{proof}
Observe that $H^{(n+1)}=H^{(n)}H_i$, where $i={\infty_{n+1}}$ is the largest index such that the pole $-\delta_n/\gamma_n$ of $h_{\infty_0}\circ\cdots\circ h_{\infty_{n}}$ is larger than the left endpoint of $I_i^\lambda$, that is 
$-\delta_n/\gamma_n>h_i(0)=P_{i}(\lambda)/P_{i+1}(\lambda)$. (See Figure~\ref{Fig:branche infinie}.) Considering the lower right coefficient of the product $H^{(n)}H_i$, we get
$$ \delta_{n+1} =  \gamma_n P_i(\lambda) + \delta_n P_{i+1(\lambda)}. $$
If $i=i_\lambda$ or if $i=0$, $0<P_{i+1}(\lambda)\le 1$ by~\eqref{eq:P_j inequalities}. 
Since $\gamma_n\le 0$ by Lemma~\ref{Lemma:coeffMatrice}, we obtain in this case $\delta_{n+1}\le \delta_n$.
On the other hand, if $0<i<i_\lambda$, the pole $-\delta_n/\gamma_n$ of $h_{\infty_0}\circ\cdots\circ h_{\infty_{n}}$ is smaller than or equal to the right endpoint $P_{i+1}(\lambda)/P_{i+2}(\lambda)$ of $I_i^\lambda$.
Since the matrix $H_i$ has determinant 1, we can write $P_{i+1}(\lambda)^2-P_i(\lambda)P_{i+2}(\lambda)=1$, which is bounded above by $P_{i+1}(\lambda)$ by~\eqref{eq:P_j inequalities}.
It follows that 
$$ \dfrac{\delta_n}{-\gamma_n} \le \dfrac{P_{i+1}(\lambda)}{P_{i+2}(\lambda)} < \dfrac{P_i(\lambda)}{P_{i+1}(\lambda)-1}\ ,$$
and $\delta_{n+1}\le \delta_n$.
\begin{figure}[h]
\input{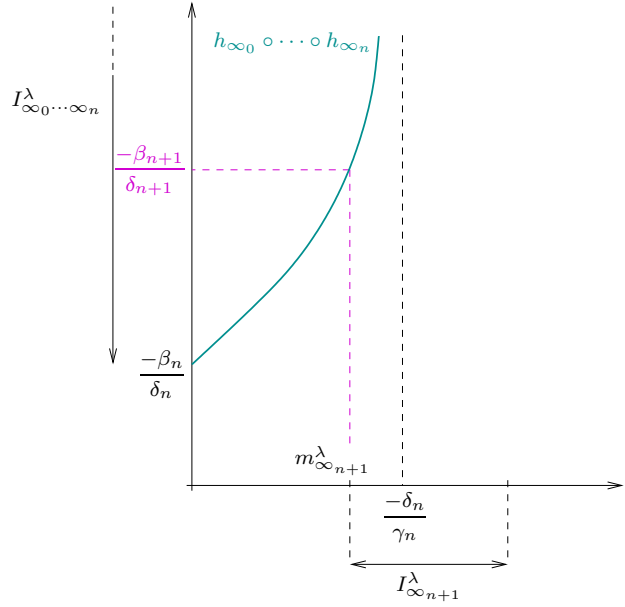}
\caption{The upper branch at order $n+1$. } 
\label{Fig:branche infinie}
\end{figure} 
\end{proof}

\subsubsection{Proof of Proposition~\ref{Prop:infinityUnreachable}}
The idea consists in proving that the slope of the upper branch $h_{\infty_0}\circ\cdots\circ h_{\infty_n}$ is always larger than 1, so that each time $\infty_{n+1}\neq0$, the left endpoint of the rightmost interval increases by at least $m_1^\lambda=1/\lambda$ (see Figure~\ref{Fig:branche infinie}).

\smallskip

Recall that the determinant of $H^{(n)}$ is equal to 1. 
Therefore, the second derivative of $h_{\infty_0}\circ\cdots\circ h_{\infty_n}$ is $-2\gamma_n/(\gamma_n x+\delta_n)^3$, which is nonnegative on $[0,-\delta_n/\gamma_n[$ by Lemma~\ref{Lemma:coeffMatrice}. 
This proves that $h_{\infty_0}\circ\cdots\circ h_{\infty_n}$ is convex on $[0,-\delta_n/\gamma_n[$.
Moreover, by Lemma~\ref{Lemma:deltadecroissant}, we get that $(h_{\infty_0}\circ\cdots\circ h_{\infty_n})'(0)=\delta_n^{-2}\ge \delta_0^{-2}=P_{i_\lambda+1}(\lambda)^{-2}\ge1$.
Hence, $(h_{\infty_0}\circ\cdots\circ h_{\infty_n})'\ge1$ on $[0,-\delta_n/\gamma_n[$.

Observe that 
$$
h_{\infty_0}\circ\cdots\circ h_{\infty_{n+1}}(0)=h_{\infty_0}\circ\cdots\circ h_{\infty_n}(m_{\infty_{n+1}}^\lambda).
$$ 
If $\infty_{n+1}\ge1$, $m_{\infty_{n+1}}^\lambda\ge m_1^\lambda=1/\lambda$. 
Hence, for all $n$ such that $\infty_{n+1}\ge1$, we get that 
$$
h_{\infty_0}\circ\cdots\circ h_{\infty_{n+1}}(0)\ge h_{\infty_0}\circ\cdots\circ h_{\infty_{n}}(0)+\frac{1}{\lambda}.
$$
But by Lemma~\ref{Lemma:zeros}, we know that there exist infinitely many such $n$'s, thus the left endpoint of $I_{\infty_0\ldots\infty_n}^\lambda$ satisfies 
$$h_{\infty_0}\circ\cdots\circ h_{\infty_{n}}(0)\tend{n}{\infty}\infty. $$
This concludes the proof of Proposition~\ref{Prop:infinityUnreachable}.$\hfill\square$

\subsection{Proof of Theorem~\ref{Thm:distinct_points}}
It remains to prove that for any $0\le x<\infty$, the length of $I_{x_0\dots x_n}^\lambda$ goes to $0$, where $x_0\dots x_n\dots =\omega_\lambda(x)$.
We have already dealt with the case when $x_n=0$ for all $n$ large enough (Lemma~\ref{lemma:zero}).
{From} now on, we assume that there exist infinitely many $n$'s such that $x_n\not=0$.

\begin{lemma}
\label{Lemma:uniform}
For all $n\ge0$, for all $z\in I_{x_0\dots x_n}^\lambda$, there exist real numbers $u_n(z)$ and $v_n(z)$ such that
\begin{equation}
\label{Eq:T_lambda(z)}
 T_\lambda^{n+1} z = \frac{u_n(z)}{v_n(z)}, \mbox{ where } 
H_{x_0}\cdots H_{x_n} \begin{pmatrix} u_n(z) \\  v_n(z) \end{pmatrix} = 
\begin{pmatrix} z \\  1 \end{pmatrix}.
\end{equation}
Moreover, there exists $(\rho_n)_n$ depending only on $\omega_\lambda(x)$ such that $0\le u_n(z) \le\rho_n$, $0<v_n(z)\le \rho_n$ and $\rho_n$ goes to $0$ as $n\to\infty$.
\end{lemma}
\begin{proof}
Let $n\ge0$ and $z\in I_{x_0\dots x_n}^\lambda$. We obtain~\eqref{Eq:T_lambda(z)} by iteration of~\eqref{eq:T_lambda} and~\eqref{eq:T_lambda matrices}. 
Recalling the geometrical interpretation in Section~\ref{Section:Geometrical Interpretation}, 
$u_n(z)$ and $v_n(z)$ can be seen as abscissae of points on a circle, whose radius $R_n(z)$ is non-increasing with $n$. Moreover the initial radius $R_0(z)$ is a continuous function of $z$. By taking $n$ large enough so that $x_0\cdots x_n\prec \infty_0\cdots\infty_n$ (application of Proposition~\ref{Prop:infinityUnreachable}), $I_{x_0\dots x_n}^\lambda$ is bounded, hence $R_0(z)$ is bounded on $I_{x_0\dots x_n}^\lambda$.

Let $s_n$ be the number of $j\in \{0,\ldots,n-2\}$ such that $x_{j+1}\neq i_\lambda$ and $(x_{j+1},x_{j+2})\neq (0,0)$. Since we assumed that the orbit of $x$ does not end with infinitely many $0$'s, and since it cannot end with infinitely many $i_\lambda$'s (consequence of Proposition~\ref{Prop:infinityUnreachable}), we get $s_n\tend{n}{\infty}\infty$. For each $z\in I_{x_0\dots x_n}^\lambda$, the number of times the hypothesis of Proposition~\ref{Prop:R decroit} are fulfilled up to time $n$ is $s_n$, and whenever they are fulfilled, the radius is divided by at least $K(\lambda)>1$. We thus get the announced result with
$$ \rho_n \egdef \sup_{z\in I_{x_0\dots x_n}^\lambda} R_0(z) \ K(\lambda)^{-s_n}. $$
\end{proof}

We say that a finite sequence $a_0,\dots, a_n$ is a \emph{standard block} if $a_i=\infty_i$ for all $i<n$ and $a_n<\infty_n$.
By Proposition~\ref{Prop:infinityUnreachable}, $\omega_\lambda(x)$ can be decomposed into standard blocks in a unique way. The interest of standard blocks is enlightened by the following result:

\begin{lemma}
\label{Lemma:matrice associee bloc standard}
If $a_0,\dots, a_n$ is a standard block, then the matrix $H_{a_0}\cdots H_{a_n}$ associated to $h_{a_0}\circ\cdots\circ h_{a_n}$ has nonnegative coefficients. 
Moreover, $I_{a_0\dots a_n}^\lambda=h_{a_0}\circ\cdots\circ h_{a_n}([0,\infty[)$.
\end{lemma}
\begin{proof}
We consider the homographic function $h_{a_0}\circ\cdots\circ h_{a_{n-1}}=h_{\infty_0}\circ\cdots\circ h_{\infty_{n-1}}$ which is the upper branch at order $n$. 
Since its pole lies in $I_{\infty_n}^\lambda$, this function is bounded on $I_{a_n}^\lambda$, which means that $h_{a_0}\circ\cdots\circ h_{a_{n}}$ is bounded on $[0, \infty[$.
Moreover, since $a_n<\infty_n\le i_\lambda$,
$$I_{a_0,\dots, a_n}^\lambda= h_{\infty_0}\circ\cdots\circ h_{\infty_{n-1}}(I_{a_n}^\lambda) 
= h_{\infty_0}\circ\cdots\circ h_{\infty_{n-1}}\bigl(h_{a_n}([0,\infty[)\bigr) .$$
In particular, $I_{a_0,\dots, a_n}^\lambda\not=\emptyset$, thus, by Lemma~\ref{Lemma:coeffMatrice}, we know that 
$H_{a_0}\cdots H_{a_n}$ is of the form 
$\begin{pmatrix}
\alpha & \beta\\
\gamma & \delta
\end{pmatrix}$
where $\beta\ge0$ and $\delta >0$. 
Since $\frac{\alpha x+ \beta}{\gamma x+ \delta}$ is bounded on $[0, \infty[$, we get $\gamma\ge0$. 
Moreover $$\frac{\alpha}{\gamma}=\lim_{x\to\infty}\uparrow \frac{\alpha x+ \beta}{\gamma x+ \delta} >0,$$ hence $\alpha>0$. 
\end{proof}

We are now ready to achieve the proof of Theorem~\ref{Thm:distinct_points}.
Consider $n$ such that $x_0\dots x_n$ is a concatenation of standard blocks. 
By the previous Lemma, the matrix $H_{x_0}\cdots H_{x_n}$ is of the form 
$$H_{x_0}\cdots H_{x_n}=
\begin{pmatrix}
\alpha & \beta\\
\gamma & \delta
\end{pmatrix},$$
where $\alpha\delta-\gamma\beta=1$, all the coefficients are nonnegative, and 
$$I_{x_0\dots x_n}^\lambda=h_{x_0}\circ\cdots\circ h_{x_{n}}([0, \infty[)=\left[\frac{\beta}{\delta}, \frac{\alpha}{\gamma}\right[.$$
Observe that the length of $I_{x_0\dots x_n}^\lambda$ is $1/\gamma\delta$. Our purpose is to prove that $\gamma\delta$ is large if $n$ is large enough.

By Lemma~\ref{Lemma:uniform}, we have $1=\gamma u_n(z) + \delta v_n(z)$ for any $z\in I_{x_0\dots x_n}^\lambda$.

Choose $z=\beta/\delta$, so that $u_n(z)=0$. Then $\delta=1/v_n(z)\ge 1/\rho_n$.

Choose now $z_j=\alpha/\gamma - \epsilon_j$ with $\epsilon_j\to0$. By compacity, we can assume that $(\epsilon_j)$ is chosen such that $u_n(z_j)$ and $v_n(z_j)$ converge, respectively to $u_n$ and $v_n$. 
Since $\lim_{j\to\infty}u_n(z_j)/v_n(z_j)=\infty$, we have $v_n=0$.
Therefore, $\gamma=1/u_n\ge 1/\rho_n$.
This concludes the proof of Theorem~\ref{Thm:distinct_points}.
$\hfill\square$

\subsection{Convergence of $\lambda$-continued fractions}
\label{Sec:continued fractions}
The above results can be interpreted in terms of convergence of $\lambda$-continued fractions. 
Starting from $x\in [0,\infty[$, we define, via the coding of its orbit $\omega_\lambda(x)=x_0 x_1\ldots$, a sequence of real numbers having finite expansions in $\lambda$-continued fractions:
For any $j\ge0$, we consider the left endpoint $m_{x_0 \ldots x_n}^\lambda\egdef h_{x_0}\circ h_{x_1}\circ \cdots \circ h_{x_n}(0)$ of the interval $I_{x_0, \dots, x_n}^\lambda$. We can recursively construct a finite expansion in $\lambda$-continued fractions of these endpoints by observing that, if $y=[0,b_1,\dots, b_\ell]_{\lambda}$, then 
$$
h_i(y)=[0,\underbrace{1,-1, \dots,(-1)^{i-1}}_{i\mbox{ terms}}, (-1)^i(1+b_1), (-1)^ib_2, \dots, (-1)^ib_\ell]_{\lambda}\ .
$$
If $y=0$, and its expansion is $[0]_\lambda$, the preceding formula is to be understood as $h_i(0)=[0,1,-1, \dots,(-1)^{i-1}]_{\lambda}$.

We obtain in this way finite expansions in $\lambda$-continued fractions, where the first coefficient is zero, the second one is positive, and the signs of the following coefficients alternate.
It will therefore be useful to introduce the following notation which corresponds to this particular form of $\lambda$-continued fractions: For positive integers $(b_i)_{1\le i\le \ell}$, we set
$$
\brg b_1, b_2, \ldots , b_\ell \brd_\lambda 
\egdef [0,b_1, -b_2, \ldots, (-1)^{\ell-1}b_\ell]_\lambda 
= \cfrac{1}{b_1\lambda - \cfrac{1}{\ddots-\cfrac{1}{b_\ell \lambda }}}
$$
With our new notation, we have 
\begin{equation}
\label{Eq:h_i}
 h_i \bigl(\brg b_1, b_2, \ldots , b_\ell \brd_\lambda \bigr)
=\brg\ \underbrace{1,\dots,1}_{i\mbox{ terms}}, (1+b_1), b_2, \dots, b_\ell\ \brd_{\lambda}\ .
\end{equation}

\begin{lemma}
\label{Lemma:left endpoint}
Consider the prefix $x_0\dots x_n$ of the orbit $\omega_\lambda(x)$ and write this finite sequence as $0^{e_0}a_00^{e_1}a_1\dots 0^{e_\ell}a_\ell 0^{e_{\ell+1}}$, where $a_i>0$ and $e_i\ge0$.
Then the left endpoint $m_{x_0 \ldots x_n}^\lambda$ of $I_{x_0\dots x_n}^\lambda$ satisfies
$$
m_{x_0 \ldots x_n}^\lambda = \brg \ \underbrace{e_0+1, 1, \ldots, 1}_{a_0\mbox{ terms}},\  \underbrace{e_1+2, 1, \ldots, 1}_{a_1\mbox{ terms}},\ \dots ,\  \underbrace{e_\ell +2, 1, \ldots, 1}_{a_\ell\mbox{ terms}}\ \brd_\lambda \ .
$$
\end{lemma}
\begin{proof}
 It is an immediate induction using~\eqref{Eq:h_i}.
\end{proof}
Note that the number of coefficients in the expansion of $m_{x_0 \ldots x_n}^\lambda$  obtained in this way is $x_0+\cdots+x_n$, and that the expansion of $m_{x_0\dots x_n}^\lambda$ is a prefix of the expansion of $m_{x_0\dots x_{n+1}}^\lambda$. 

Now, Theorem~\ref{Thm:distinct_points} can be interpreted as the convergence to $x$ of $m_{x_0 \ldots x_n}^\lambda$. 
This allows to write $x$ in the form of an infinite expansion in $\lambda$-continued fraction:
$$
x = \brg \ \underbrace{e_0+1, 1, \ldots, 1}_{a_0\mbox{ terms}},\  \underbrace{e_1+2, 1, \ldots, 1}_{a_1\mbox{ terms}},\ \dots ,\  \underbrace{e_\ell +2, 1, \ldots, 1}_{a_\ell\mbox{ terms}},\ \dots \ \brd_\lambda \ .
$$

Moreover, we can deduce from the proof of Theorem~\ref{Thm:distinct_points} the speed of convergence of $m_{x_0 \ldots x_n}^\lambda$ to $x$, which is essentially given by the combinatorics of the sequence $\omega_\lambda(x)=x_0x_1\ldots x_n\ldots$: If there exists some $n_0$ such that $x_n=0$ for each $n\ge n_0$, then $x= m_{x_0 \ldots x_{n_0}}^\lambda$; Otherwise, the proof of Theorem~\ref{Thm:distinct_points} yields 
$$ \left| x - m_{x_0 \ldots x_{n}}^\lambda \right| \le \left| I_{x_0 \ldots x_{n}}^\lambda \right| \le \rho_n^2, $$
where $\rho_n$ is defined in Lemma~\ref{Lemma:uniform}.

\section{Link with $\beta$-shifts}

\subsection{Conjugacy between $T_\lambda$ and some $\beta$-shift}
\label{Sec:conjugacy}
\subsubsection{Classical results on $\beta$-shifts}
\label{Sec:beta-shifts}
For any $\beta>1$, Renyi~\cite{renyi1957} introduced the $\beta$-expansion of a real number $0\le t<1$ as the series $$
t = \sum_{n\ge 0}\frac{t_{n}}{\beta^{n+1}}\ ,
$$
where the coefficients $t_n$ are nonnegative integers defined as follows. 
We consider the transformation $S_\beta:[0,1[\to[0,1[$, which sends $t$ to $\beta t\mod 1$; 
The natural partition of $[0,1[$ associated to this transformation is composed of intervals of the form 
$J_j^\beta\egdef\left[\frac{j}{\beta}, \frac{j+1}{\beta}\right[$, for $0\le j<\beta-1$, together with 
$J_{\lfloor\beta\rfloor}^\beta\egdef\left[\frac{\lfloor\beta\rfloor}{\beta}, 1\right[$, where $\lfloor\beta\rfloor$ denotes the largest integer smaller than $\beta$;
The coefficients $t_n$ are given by the orbit of $t$ under the transformation $S_\beta$: $t_n=j$ if $S_\beta^{n}(t)\in J_j^\beta$.

\begin{figure}[h]
\input{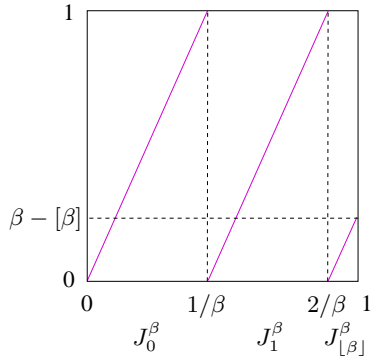}
\caption{Graph of $S_\beta$ for $\beta=9/4$.} 
\label{Fig:beta shift}
\end{figure} 

We denote by $\O_\beta(t)$ the sequence $(t_n)_{n\ge0}$ coding the orbit of $t\in [0, 1[$ under the transformation $S_\beta$. 
As in the case of $T_\lambda$, the map $t\mapsto \O_\beta(t)$ is increasing with respect to the lexicographic order of sequences. 
We denote by $\O_\beta(1)$ the limit, as $t\to 1$, of $\O_\beta(t)$. 
Observe that for any $n\ge0$, $\sigma^n \O_\beta(1)\preceq \O_\beta(1)$.

Parry characterized the sequences which code orbits for $S_\beta$ using the $\beta$-expansion of the fractional part $\beta-[\beta]$ of $\beta$. For our purpose, it is useful to translate his result in terms of $\O_\beta(1)$, which can be done by the following lemma.

\begin{lemma}
 \label{Lemma:orbite1}
If $\beta\in\ZZ_+$, $\O_\beta(1)=(\beta -1)(\beta -1)\dots$.
If $\beta\notin\ZZ_+$, let $(b_i)_{i\ge0}$ be the coding of the orbit of $\beta-[\beta]$.
If $(b_i)_i$ contains infinitely many nonzero terms, then $\O_\beta(1)=[\beta]b_0b_1\dots$. 
Otherwise, $\O_\beta(1)$ is a periodic repetition of the pattern $[\beta]b_0\dots b_{\ell-1}(b_\ell -1)$, where $b_\ell$ is the last nonzero term.
\end{lemma}
\begin{proof}
 The case where $\beta$ is an integer is clear. We now assume that $\beta\notin\ZZ$. Then
$$ \lim_{t\uparrow 1}\uparrow S_\beta(t) =\beta-[\beta].$$
 If $(b_i)_i$ contains infinitely many nonzero terms, the orbit of $\beta-[\beta]$ never meets 0, hence $S_\beta^n$ is continuous at $\beta-[\beta]$. We thus have
$$ \lim_{t\uparrow \beta-[\beta]}\uparrow \O_\beta(t) = \O_\beta(\beta-[\beta]).$$
This proves that $\O_\beta(1)=[\beta]b_0b_1\dots$ in this case.

On the other hand, if there exists a smaller integer $\ell\ge0$ such that $b_n=0$ for all $n>\ell$, then $S_\beta^{\ell+1}(\beta-[\beta])=0$ and $S_\beta^{\ell}(\beta-[\beta])=b_\ell/\beta$, where $b_\ell\ge1$. Therefore,
$$ \lim_{t\uparrow \beta-[\beta]}\uparrow S_\beta^{\ell+1}(t) = 1,\quad \mbox{and}\quad 
    \lim_{t\uparrow 1}\uparrow S_\beta^{\ell+2}(t) = 1. $$
Then $\O_\beta(1)$ is periodic, with period $\ell+2$, and it is easily checked that the periodic pattern is $[\beta]b_0\dots b_{\ell-1}(b_\ell -1)$.
\end{proof}

\begin{theo}[Parry~\cite{parry1960}, Theorem 3]
\label{Thm:Parry}
A sequence $(t_n)_{n\ge0}$ is the $\beta$-expansion of a real number $t\in [0, 1[$ if and only if 
\begin{equation}
 \label{Eq:o_beta(1)_admissible}
t_n t_{n+1}\dots \prec \O_\beta(1), \quad \forall n\ge0.
\end{equation}
\end{theo}

\begin{remark}
\label{rmk:nonperiodic}
Observe that $\O_\beta(1)$ cannot be periodic if the orbit of $\beta-[\beta]$ never meets 0. Indeed, if it were, we would have some $n\ge0$ such that
$$ \sigma^n \O_\beta(\beta-[\beta]) = \O_\beta(1), $$
which would contradict the theorem.
\end{remark}

\subsubsection{Characterization of orbits of $T_\lambda$}
As in the case of $\beta$-expansions, we can characterize the sequences which code orbits for $T_\lambda$.
\begin{theo}
\label{Thm:characterization}
A sequence $(x_n)_{n\ge0}$ codes the orbit of a real number $x$ under the action of $T_\lambda$ if and only if 
\begin{equation}
 \label{Eq:omega_lambda(infty)_admissible}
x_n x_{n+1}\dots \prec \omega_\lambda(\infty), \quad \forall n\ge0.
\end{equation}
\end{theo}
\begin{proof}
The condition is necessary by Proposition~\ref{Prop:infinityUnreachable}. 
Conversely, consider a sequence $(x_n)_{n\ge0}$ satisfying $x_n x_{n+1}\dots \prec \omega_\lambda(\infty)$ for any $n\ge0$. 
This means that this sequence can be decomposed into a concatenation of standard blocks $x_0x_1\dots=B_0B_1\dots$. 
For any standard block $B=b_0\dots b_\ell$, let $h_{B}\egdef h_{b_0}\circ \cdots \circ h_{b_\ell}$.
By Lemma~\ref{Lemma:matrice associee bloc standard}, for each $k\ge0$, the interval $[u_k, v_k[\egdef I_{B_0\dots B_k}^\lambda$ is equal to $h_{B_0}\circ \dots h_{B_k}([0,\infty[)$. 
Therefore, $v_k=h_{B_0}\circ \dots h_{B_k}(\infty)$, and 
$v_{k+1}=h_{B_0}\circ \dots h_{B_k}(h_{B_{k+1}}(\infty))$. 
Since $h_{B_{k+1}}(\infty)<\infty$, we get that $v_{k+1}<v_k$. 
Hence 
$$
\bigcap_n I_{x_0\dots x_n}^\lambda=\bigcap_k I_{B_0\dots B_k}^\lambda = \bigcap_k [u_k, v_k[ = \bigcap_k [u_k, v_k] \neq \emptyset .
$$
\end{proof}

Since the sets of sequences coding orbits for $T_\lambda$ and for $S_\beta$ have the same structure, we are led to find a correspondence between $\lambda$ and $\beta$: In view of the statements of Theorems~\ref{Thm:Parry} and~\ref{Thm:characterization}, given $\lambda\in]0,2[$, we want to find $\beta>1$ such that $\O_\beta(1)=\omega_\lambda(\infty)$.

\subsubsection{How to find $\beta(\lambda)$}
The main tool to find $\beta$ is the following result by Parry.
\begin{prop}[Parry~\cite{parry1960}, Corollary 1]
Let $b$ be a positive integer and let $(b_i)_{i\ge0}$ be a sequence of nonnegative integers. 
There exists $\beta>1$ with integer part $b$ such that the $\beta$-expansion of $(\beta-[\beta])$ is coded by $(b_i)_{i\ge0}$ if and only if 
$$
b_nb_{n+1}\dots \prec b\,b_0b_1\dots, \mbox{ for all } n\ge0.
$$
\end{prop}
It also follows from Parry that $\beta$ is unique whenever it exists.

By Lemma~\ref{lemma:lexico}, $\sigma^n \omega_\lambda(\infty)\preceq \omega_\lambda(\infty)$ for all $n\ge1$. 
If $\omega_\lambda(\infty)=\infty_0\infty_1\dots$ is non periodic, then the previous inequality is strict. Therefore, there exists a unique  $\beta>1$ with integer part $\infty_0$ such that the $\beta$-expansion of $(\beta-[\beta])$ is coded by $(\infty_i)_{i\ge1}$.
By Lemma~\ref{Lemma:zeros} and Lemma~\ref{Lemma:orbite1}, we conclude that $\O_\beta(1)=\omega_\lambda(\infty)$.
It remains to study the case where $\omega_\lambda(\infty)$ is periodic. 
Let $p$ be the smallest integer such that $\omega_\lambda(\infty)$ is a periodic repetition of the pattern $\infty_0\dots\infty_p$. 
We then define $\bar\omega\egdef \infty_0\dots\infty_{p-1}(\infty_p+1)00\dots$. 
We let the reader check that $\sigma^n \bar\omega\prec \bar\omega$ for all $n\ge1$. Therefore, there exists a unique $\beta>1$ with integer part $\infty_0$ ($\infty_0+1$ if $p=0$) such that the $\beta$-expansion of $(\beta-[\beta])$ is coded by $\infty_1\dots\infty_{p-1}(\infty_p+1)00\dots$.
By Lemma~\ref{Lemma:orbite1}, we conclude that $\O_\beta(1)=\omega_\lambda(\infty)$.

Eventually, we proved that for any $\lambda\in ]0,2[$, there exists a unique $\beta>1$ such that $\O_\beta(1)=\omega_\lambda(\infty)$.

\bigskip
By Theorem~\ref{Thm:characterization}, $x\mapsto \omega_\lambda(x)$ is a one-to-one map from $[0,\infty[$ onto the set of integer-valued sequences satisfying~\eqref{Eq:omega_lambda(infty)_admissible}.
In the same way, by Theorem~\ref{Thm:Parry}, $t\mapsto \O_\beta(t)$  is a one-to-one map from $[0,1[$ onto the set of integer-valued sequences satisfying~\eqref{Eq:o_beta(1)_admissible}.
If $\beta$ is such that $\O_\beta(1)=\omega_\lambda(\infty)$, we get by composition a one-to-one map $\phi_\lambda$ sending $x\in[0,\infty[$ to $t\in[0,1[$, where $\O_\beta(t)=\omega_\lambda(x)$.
The following diagram commutes:

\begin{figure}[h]
\input{diagram.pstex_t}
\caption{} 
\label{Fig:diagram}
\end{figure}

Moreover, $\phi_\lambda$ is increasing, and since it is onto, $\phi_\lambda$ is continuous.

\begin{theo}
 \label{Thm:correspondence}
For any $\lambda\in ]0,2[$, there exists a unique $\beta=\beta(\lambda)>1$ and a homeomorphism $\phi_\lambda:[0,\infty[\to[0,1[$ conjugating $\bigl([0, \infty[, T_\lambda\bigr)$ and $\bigl([0, 1[,S_\beta\bigr)$.
\end{theo}

We observed in~\cite{janvresse2008} some connection, when $\lambda=1$, with Minkowski's question mark function. Translated in the context of the present paper, this connection can be written as follows: 
$$ \phi_1(x) = \begin{cases}
\frac{1}{2} ?(x) & \text{ if } 0\le x\le 1,\\
\frac{1}{2}+\frac{1}{2} ?(1-1/x) & \text{ otherwise. }
               \end{cases}
$$

Since it is well known that the topological entropy of the $\beta$-shift is $\log\beta$ (see \textit{e.g.}~\cite{ito1974}), we get the following corollary of Theorem~\ref{Thm:correspondence}:
\begin{corollary}
 For any $\lambda\in ]0,2[$, the topological entropy of $T_\lambda$ is equal to $\log\beta(\lambda)$, where $\beta(\lambda)$ is defined in Theorem~\ref{Thm:correspondence}.
\end{corollary}

\subsection{Properties of $\lambda\mapsto\beta(\lambda)$}
\label{Sec:map}

The purpose of this section is to prove the following theorem:
\begin{theo}
\label{Thm:beta(lambda)}
 The map $\lambda\mapsto\beta(\lambda)$ is increasing and continuous from $]0,2[$ onto $]1,\infty[$.
\end{theo}

\subsubsection{Variation of $\lambda\mapsto\beta(\lambda)$}

\begin{lemma}
\label{Lemma:variation des bornes}
Let $W$ be a finite sequence of nonnegative integers. The left endpoint $m_W^{\lambda}$ of $I_W^\lambda$ is a continuous decreasing function of $\lambda$ on its interval of definition. 
\end{lemma}
\begin{proof}
Lemma~\ref{Lemma:left endpoint} gives the explicit expression of $m_W^{\lambda}$ in terms of $\lambda$-continued fractions. 
It is easily checked by induction on $\ell$ that $\lambda\mapsto \brg a_1, a_2,\ldots , a_\ell \brd_\lambda$ is a decreasing function of $\lambda$. 
\end{proof}

\begin{lemma}
\label{Lemma:variation a x fixe}
Let $0<\lambda<\lambda'<2$. If $x\ge0$ is such that there exists $a_0,\ldots,a_{n-1}$ satisfying $x\in I_{a_0\ldots a_{n-1}}^\lambda\cap I_{a_0\ldots a_{n-1}}^{\lambda'}$, then $T_\lambda^n(x)<T_{\lambda'}^n(x)$.
\end{lemma}
\begin{proof}
For $n=1$, we get the result by an easy induction on $a_0$. We then make an induction on $n$ to prove the lemma.
\end{proof}

\begin{lemma}
\label{Lemma:variation2}
For a fixed $x\in[0,\infty[$, if $0<\lambda<\lambda'<2$, then $\omega_\lambda(x)\prec\omega_{\lambda'}(x)$.
\end{lemma}
\begin{proof}
Let $\omega_\lambda(x)=a_0\ldots a_{n}\ldots$, and $\omega_{\lambda'}(x)=a_0'\ldots a_{n}'\ldots$. Assume that for some $n\ge1$, $a_0\ldots a_{n-1}=a_0'\ldots a_{n-1}'$. Then by Lemma~\ref{Lemma:variation a x fixe}, $T_\lambda^n(x)<T_{\lambda'}^n(x)$. But $T_\lambda^n(x)\in I_{a_n}^\lambda=[m_{a_n}^\lambda,m_{a_n+1}^\lambda[$, while $T_{\lambda'}^n(x)\in I_{a_n'}^{\lambda'}=[m_{a_n'}^{\lambda'},m_{a_n'+1}^{\lambda'}[$, thus $m_{a_n}^{\lambda}< m_{a_n'+1}^{\lambda'}$. By Lemma~\ref{Lemma:variation des bornes}, $m_{a_n}^{\lambda'}\le m_{a_n}^{\lambda}$, therefore $m_{a_n}^{\lambda'}<m_{a_n'+1}^{\lambda'}$. This proves that $a_n\le a'_n$. It follows that $\omega_\lambda(x)\preceq\omega_{\lambda'}(x)$. It remains to show that the orbits are different. If $a_0=a'_0$, by Lemma~\ref{Lemma:variation a x fixe}, we get $T_\lambda(x)<T_{\lambda'}(x)$. Hence by Theorem~\ref{Thm:distinct_points}, $\omega_{\lambda'}\left(T_\lambda(x)\right) \prec\omega_{\lambda'}\left(T_{\lambda'}(x)\right)$. But we already know that $\omega_{\lambda}\left(T_\lambda(x)\right) \preceq\omega_{\lambda'}\left(T_{\lambda}(x)\right)$.
\end{proof}

\begin{prop}
\label{Prop:increasing}
 $\lambda\mapsto\beta(\lambda)$ is increasing.
\end{prop}
\begin{proof}
 By Lemma~3 in \cite{parry1960}, it is enough to show that for $\lambda<\lambda'$, $\omega_\lambda(\infty)\prec\omega_{\lambda'}(\infty)$. Since $\omega_{\lambda}(\infty)$ and $\omega_{\lambda'}(\infty)$ respectively start with $i_\lambda$ and $i_{\lambda'}$, we can assume that there exists $k\ge1$ such that $\lambda_{k+1}<\lambda<\lambda'\le\lambda_{k+2}$, so that $i_\lambda=i_{\lambda'}=k$ (otherwise, we have $i_\lambda<i_{\lambda'}$ which directly gives the result).

We first show that $\ell_\lambda<\ell_{\lambda'}$. If $\lambda'=\lambda_{k+2}$, then $\ell_\lambda<\infty=\ell_{\lambda'}$. Otherwise, $\ell_{\lambda'}= T_{\lambda'}(\lambda')$. By Lemma~\ref{Lemma:variation des bornes}, we have
$$ m_{k-1}^{\lambda'} < m_{k-1}^{\lambda} < \lambda < \lambda' \le m_k^{\lambda'} < m_k^\lambda, $$
hence $\lambda\in I_{k-1}^{\lambda}\cap I_{k-1}^{\lambda'}$. Then by Lemma~\ref{Lemma:variation a x fixe}, $\ell_{\lambda'}= T_{\lambda'}(\lambda')>T_{\lambda'}(\lambda)>T_{\lambda}(\lambda)=\ell_\lambda$. 

Observe now that, since $\ell_{\lambda'}>\ell_\lambda$,
$$\sigma\omega_{\lambda'}(\infty)=\lim_{x\uparrow\ell_{\lambda'}}\uparrow \omega_{\lambda'}(x)\succ
\omega_{\lambda'}(\ell_\lambda),$$
which by Lemma~\ref{Lemma:variation2} is lexicographically after $\omega_{\lambda}(\ell_\lambda)$. The conclusion follows by noting that $\omega_{\lambda}(\ell_\lambda)\succeq\lim_{x\uparrow\ell_{\lambda}}\uparrow \omega_{\lambda}(x)=\sigma\omega_{\lambda}(\infty)$.
\end{proof}

\subsubsection{Surjectivity of $\lambda\mapsto\beta(\lambda)$}
\newcommand{\lsm}{\text{LSM}}
We define the set of sequences which are lexicographically shift maximal (LSM) as:
$$ \lsm \egdef \left\{ W=x_0 x_1\cdots\in\ZZ_+^{\ZZ_+}:\ x_0>0 \mbox{ and } \forall k\ge 0,\ x_k x_{k+1}\cdots \preceq  x_0 x_1\cdots \right\}. $$
By Lemma~\ref{lemma:lexico}, for any $0<\lambda<2$, $\omega_\lambda(\infty)\in \lsm$.

In the same way, we define the set of words of length $n+1$ which are lexicographically shift maximal as:
$$ \lsm_n \egdef \left\{ W=x_0\cdots x_n\in\ZZ_+^{n+1}:\ x_0>0 \mbox{ and }  \forall 0\le k\le n,\ x_k\cdots x_n\preceq  x_0\cdots x_n \right\}. $$

\begin{lemma}
 \label{Lemma:successeur dans lsm_n}
Let $W=x_0\cdots x_n\in\lsm_n$. 
Then $\SUCC(W)\egdef\min\{W'\in\lsm_n: W'\succ W\}$ always exists and is obtained in the following way:
Consider the longest strict suffix $x_{k+1}\dots x_n$ of $W$ which is also a prefix of $W$. Then 
$$\SUCC(W)=
\begin{cases}
 x_0\dots x_{k-1}(x_{k}+1)\!\!\!\!\underbrace{0\dots 0}_{n-k \mbox{ terms}}& \mbox{ if }k\not=0\\
(x_0+1)  \underbrace{0\dots 0}_{n \mbox{ terms}}& \mbox{ if }k=0.
\end{cases}
$$
\end{lemma}
\begin{proof}
 easy exercise !!!
\end{proof}

For any $\lambda\in]0,2[$, let $\MAX_n^\lambda$ be the prefix of length $n+1$ of $\omega_\lambda(\infty)$:
$\MAX_n^\lambda\in\lsm_n$ and by Lemma~\ref{Lemma:variation2} $\lambda\mapsto \MAX_n^\lambda$ is non-decreasing.

\begin{lemma}
 \label{Lemma:interval}
Let $W\in\lsm_n$. If the set $\left\{\lambda\in ]0,2[: \MAX_n^\lambda=W\right\}$ is nonempty, then it is an interval of the form $\left]\lambda_W^{\min}, \lambda_W^{\max}\right]$.
\end{lemma}
\begin{proof}
Since  $\lambda\mapsto \MAX_n^\lambda$ is non-decreasing, the above set is an interval. 
Let $\lambda$ be such that $\MAX_n^{\lambda}=W$ and consider $x>m_W^{\lambda}$. 
By Lemma~\ref{Lemma:variation des bornes}, for $\lambda'<\lambda$ close enough to $\lambda$, $x>m_W^{\lambda'}$, thus $\MAX_n^{\lambda'}\succeq W$. On the other hand, $\lambda'\mapsto \MAX_n^{\lambda'}$ is non-decreasing. 
We conclude that $\MAX_n^{\lambda'}=W$ for all $\lambda'<\lambda$ close enough to $\lambda$.
\end{proof}

We denote by $\mathring{I}$ the interior of $I$. 
\begin{lemma}
 \label{Lemma:lien entre orbites}
Let $\lambda$ be such that $\omega_\lambda(\lambda)$ starts with $x_0\dots x_n$.
If $\lambda\in\mathring{I}^\lambda_{x_0\dots x_n}$, then $\MAX_n^{\lambda'}=(x_0+1)x_1\dots x_n$ for any $\lambda'$ close enough to $\lambda$.
\end{lemma}
\begin{proof}
Since $\lambda\in\mathring{I}^\lambda_{x_0}$, we have $x_0=i_\lambda-1$, thus $\omega_\lambda(\infty)$ starts with $i_\lambda=x_0+1$.
Moreover, $\lambda\not=m_{i_\lambda}^\lambda$ implies that $\ell_\lambda=T_\lambda(\lambda)$. 
Hence, $\ell_\lambda\in\mathring{I}^\lambda_{x_1}$, $T_\lambda(\ell_\lambda)\in\mathring{I}^\lambda_{x_2}$, \dots, $T_\lambda^{n-1}(\ell_\lambda)\in\mathring{I}^\lambda_{x_n}$. It follows that $\MAX_n^{\lambda}=(x_0+1)x_1\dots x_n$ because $\lim_{x\to\infty}T_\lambda(x)=\ell_\lambda$.
By continuity with respect to $\lambda$ of the endpoints of the interval $I^\lambda_{x_0\dots x_n}$, any $\lambda'$ close enough to $\lambda$ also satisfies the hypothesis of the lemma and the claim follows.
\end{proof}

\begin{lemma}
 \label{Lemma:successeur}
Let $W\in\lsm_n$ such that $\left\{\lambda\in ]0,2[: \MAX_n^\lambda=W\right\}$ is nonempty. 
For any $\lambda'>\lambda_W^{\max}$ close enough to $\lambda_W^{\max}$, $\MAX_n^{\lambda'}=\SUCC(W)$.
\end{lemma}
\begin{proof}
Let $\lambda\egdef \lambda_W^{\max}$ and $x_0 x_1\dots \egdef\omega_{\lambda}(\lambda)$.
By definition, $\lambda=\lambda_W^{\max}$ does not satisfy the conclusion of Lemma~\ref{Lemma:lien entre orbites}. Therefore, $\lambda\notin\mathring{I}^\lambda_{x_0\dots x_n}$. 
Consider the smallest $j\ge0$ such that $\lambda\notin\mathring{I}^\lambda_{x_0\dots x_j}$.

If $j=0$, $x_0=i_\lambda$ and $\lambda=m_{i_\lambda}^\lambda$. In this case, we have $\lim_{x\to\infty}T_\lambda(x)=\infty$, thus $\omega_\lambda(\infty)=i_\lambda i_\lambda\dots$ and $W=i_\lambda \dots i_\lambda$.

Assume now that $j\ge1$. 
Since $\lambda\in\mathring{I}^\lambda_{x_0\dots x_{j-1}}$, Lemma~\ref{Lemma:lien entre orbites} proves that, for all $\lambda'$ close enough to $\lambda$, $\MAX_{j-1}^{\lambda'}=(x_0+1)x_1 \dots x_{j-1}$.
On the other hand, $T_\lambda^j(\lambda)=b^\lambda_{x_j}$, $x_j\not=0$ and $x_{j+1}=\dots=x_n=0$. 
We have in this case $\ell_\lambda=T_\lambda(\lambda)$. 
Thus, $\lim_{x\uparrow\ell_\lambda}\uparrow T_\lambda^{j-1}(x)=b^\lambda_{x_j}$ and 
$\lim_{x\uparrow\ell_\lambda}\uparrow T_\lambda^{j}(x)=\infty$.
Since $\lim_{x\to\infty}\uparrow T_\lambda(x)=\ell_\lambda$, we get $\lim_{x\uparrow\infty}\uparrow T_\lambda^{j}(x)=b^\lambda_{x_j}$ and $\lim_{x\uparrow\infty}\uparrow T_\lambda^{j+1}(x)=\infty$.
This means that the $(j+1)$-th term of $\omega_{\lambda}(\infty)$ is $x_j-1$ and $\omega_{\lambda}(\infty)$ is periodic of period $j+1$. This proves that $W$ is the prefix of length $n+1$ of the periodic repetition of the pattern $(x_0+1)x_1 \dots x_{j-1}(x_j-1)$.
Moreover, $j+1$ is the smallest period for $\omega_{\lambda}(\infty)$. Indeed, if we had a smallest period $r$, we would have $\lim_{x\to\infty}T_\lambda^r(x)=\infty$, which would imply by a similar argument that $T_\lambda^{r-1}(\lambda)=b^\lambda_{x_{r-1}}$. This would contradict the definition of $j$.

It remains to prove that for any $\lambda'>\lambda$ close enough to $\lambda$, $\MAX_n^{\lambda'}=\SUCC(W)$, that is $(x_0+1)x_1 \dots x_{j-1}x_j 0\dots 0$ when $j\ge1$ and $(x_0+1)0\dots 0$ when $j=0$.
Observe that the prefix of length $n+1$ of $\omega_{\lambda'} (\lambda')$ is a right-continuous function of $\lambda'$ by Lemmas~\ref{Lemma:variation des bornes} and~\ref{Lemma:variation a x fixe}. 
Therefore, for any $\lambda'>\lambda$ close enough to $\lambda$, $\omega_{\lambda'} (\lambda')$ starts with $x_0x_1\dots x_j 0\dots 0$ and $\lambda'\in\mathring{I}^{\lambda'}_{x_0\dots x_{j}0\dots 0}$. Applying Lemma~\ref{Lemma:lien entre orbites} we conclude the proof.
\end{proof}

\begin{prop}
\label{Prop:lsm_n}
For all $n\ge0$ and all $W\in\lsm_n$, there exists $\lambda\in]0,2[$ such that $W$ is a prefix of $\omega_\lambda(\infty)$. 
\end{prop}
\begin{proof}
Fix $n\ge0$. The smallest sequence in $\lsm_n$ is $W=10\dots 0$. 
We first prove that $W=\MAX_n^\lambda$ for $\lambda$ small enough. 
For any $\lambda<1$, $i_\lambda=1$. Thus, $\ell_\lambda$ is the pole of $h_1$: $\ell_\lambda=\lambda(1-\lambda^2)^{-1}$ which tends to zero as $\lambda\downarrow 0$.
By Lemmas~\ref{Lemma:variation des bornes} and~\ref{Lemma:variation a x fixe}, for $\lambda$ small enough, we can ensure that $W=\MAX_n^\lambda$.

Any $W\in\lsm_n$ is such that $\{W'\in\lsm_n: W'\preceq W\}$ is finite. 
Hence, a repeated iteration of Lemma~\ref{Lemma:successeur} gives the desired result.
\end{proof}

\begin{prop}
\label{Prop:surjectivity}
 For all $W\in \lsm$, there exists $\lambda\in]0,2[$ such that $W=\omega_\lambda(\infty)$ if and only if $W$ does not end with infinitely many zeros. 
\end{prop}
\begin{proof}
By Lemma~\ref{Lemma:zeros}, $\omega_\lambda(\infty)$ cannot end with infinitely many zeros.
Conversely, let $W\in \lsm$ and for each $n\ge0$, let $W_n\in\lsm_n$ be the prefix of length $n+1$ of $W$. 
By Proposition~\ref{Prop:lsm_n}, $]\lambda_{W_n}^{\min},\lambda_{W_n}^{\max}]$ is nonempty. 
Thus, the decreasing sequence of these intervals has a nonempty intersection if $\lambda_{W_{n+1}}^{\min}>\lambda_{W_{n}}^{\min}$ infinitely often. 
Since $\lambda\mapsto\MAX_{n+1}^\lambda$ is non-decreasing, the equality $\lambda_{W_{n+1}}^{\min}=\lambda_{W_{n}}^{\min}$ is equivalent to $W_{n+1}=W_n0$.
\end{proof}

\begin{proof}[Proof of Theorem~\ref{Thm:beta(lambda)}]
 By Proposition~\ref{Prop:increasing}, $\lambda\mapsto\beta(\lambda)$ is increasing. By Proposition~\ref{Prop:surjectivity} and Lemma~\ref{Lemma:orbite1}, $\lambda\in]0,2[\mapsto\beta\in]1,\infty[$ is onto. By monotonicity, it is henceforth continuous.
\end{proof}

\subsubsection{Particular values of $\lambda\mapsto\beta(\lambda)$}

We turn in this section to the determination of $\beta$ corresponding to specific values of $\lambda$. The simplest case is when $\lambda=\lambda_k=2\cos(\pi/k)$ for some integer $k\ge3$. We know by Proposition~\ref{Prop:i_lambda} that $\lambda_k$ is the largest $\lambda$ for which $i_\lambda=k-2$. Hence $\omega_{\lambda_k}(\infty)$ is the largest $LSM$ sequence starting with $(k-2)$, that is $(k-2) (k-2) (k-2) \ldots$ The corresponding $\beta$ satisfies $\O_\beta(1)=(k-2) (k-2) (k-2) \ldots$, thus 
$$\beta(2\cos(\pi/k)) = k-1. $$

Another family of $\lambda$'s for which we determine the associated $\beta$ is $\lambda=1/\sqrt{k}$ for $k\ge2$.

\begin{lemma}
 For all $k\ge1$, if $\lambda\le1/\sqrt{k}$ the sequence $\omega_\lambda(\infty)$ starts with $1\underbrace{0\dots 0}_{k-1}$.
\end{lemma}
\begin{proof}
 We prove the result by induction on $k$. 

If $\lambda\le1$, Proposition~\ref{Prop:i_lambda} shows that $i_\lambda=1$ hence $\infty_0=1$, which proves the result for $k=1$.

Let $k\ge1$ such that the result holds for $k$. This means that for $\lambda\le1/\sqrt{k}$, the upper branch at order $k$ is $h_1\circ h_0^{k-1}$. The associated matrix is 
\begin{equation}
 \label{eq:upperBranch}
H_1H_0^{k-1} = \begin{pmatrix}
                   k\lambda & 1\\
		   k\lambda^2-1 & \lambda
                  \end{pmatrix}.
\end{equation}
Observe that the pole of $h_1\circ h_0^{k-1}$ is $\lambda/(1-k\lambda^2)\le m_1^\lambda=1/\lambda$ as soon as $\lambda\le 1/\sqrt{k+1}$. Therefore $\infty_{k+1}^\lambda=0$ if $\lambda\le 1/\sqrt{k+1}$.
\end{proof}

Note that the function $h_1\circ h_0^{k-1}$ is affine if $\lambda=1/\sqrt{k}$ by~\eqref{eq:upperBranch}. We then have $\lim_{x\uparrow\infty}\uparrow T_{1/\sqrt{k}}^k (x) = \infty$, hence $\omega_{1/\sqrt{k}}(\infty)$ is a periodic repetition of the pattern $1\underbrace{0\dots 0}_{k-1}$. It remains to find the associated $\beta$. By Lemma~\ref{Lemma:orbite1}, $\O_\beta(\beta-[\beta])$ is $\underbrace{0\dots 0}_{k-2}10\ldots 0\ldots$
Since $\beta$ corresponds to $\lambda=1/\sqrt{k}<1$, we know that $\beta<2=\beta(1)$, thus $[\beta]=1$. Hence $\beta$ satisfies 
$$ \beta-1 = \dfrac{1}{\beta^{k-1}}. $$
Since $\beta>1$, $\beta(1/\sqrt{k})$ is the largest real root of $X^k-X^{k-1}-1$.

\subsubsection{Asymptotic behaviour of $\tau\mapsto\beta(2\cos(\pi/\tau))$}

\begin{figure}
\input{beta2-21.pstex_t}\includegraphics{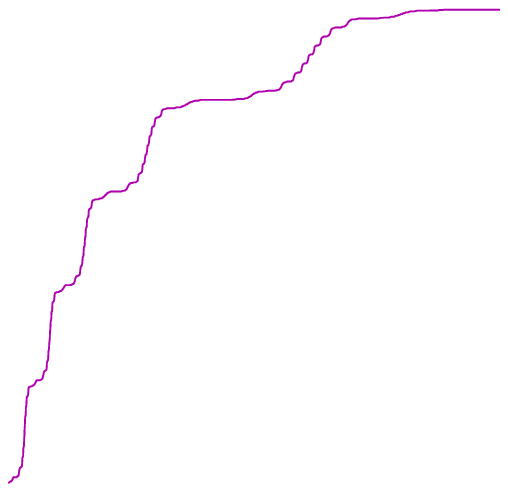}
\caption{Graph of $\tau\mapsto \beta(2\cos(\pi/\tau))$. Detail of the graph for $10.5\le \tau \le 11$.} 
\label{Fig:beta(t)}
\end{figure} 

As suggested by the particular values of $\beta$ obtained for $\lambda=2\cos(\pi/k)$, $k$ integer, $k\ge3$, we investigate here some properties of $\beta$ seen as a function of $\tau$, where $\tau>2$ is related to $\lambda$ by the relation $\lambda=2\cos(\pi/\tau)$. By composition, the map $\tau\mapsto\beta$ is also increasing, continuous, and sends $]2,\infty[$ onto $]1,\infty[$. 
We present in Figure~\ref{Fig:beta(t)} a numerical plot of this map, on which an asymptotic phenomenon appears: As $t$ grows to infinity, we see more and more abrupt steps passing from one integer to the next one, whereas $\beta$ remains almost constant when $\tau$ ranges over an interval $]k-1/2, k+1/2[$. This staircase phenomenon is proved in the following proposition.

\begin{prop}
\label{Prop:staircase}
 For all $0<\varepsilon<1/2$, we have
\begin{equation}
 \sup_{\tau\in \left]k-\frac{1}{2}+\varepsilon\, ,\, k+\frac{1}{2}-\varepsilon\right[}\ \Bigl| \beta(2\cos(\pi/\tau))- (k-1) \Bigr| \tend{k}{\infty} 0.
\end{equation}
\end{prop}

\begin{lemma}
 \label{Lemma:asymptoticOfBeta}
Let $\beta>1$ be such that $\O_\beta(1)$ starts with $k\ j$, for some integer $k\ge 2$ and some $0\le j\le k$. Then
$$ \left|\beta-\left( k+\dfrac{j}{k}\right)\right| \le \dfrac{1}{k}. $$
\end{lemma}

\begin{proof}
We use Lemma~\ref{Lemma:orbite1}. If $\O_\beta(1) = kkk\dots$, then $\beta=k+1=k+j/k$. In all the other cases, observe that $[\beta]=k$. 
If $\O_\beta(1)$ is a periodic repetition of the pattern $kj$ with $j<k$, then 
$$ \O_\beta(\beta-[\beta]) = (j+1) 0 0 0 \ldots $$
Hence we have $\beta-k=(j+1)/\beta$, which gives the result since $\beta-k\le 1$. 
Otherwise, $\O_\beta(\beta-[\beta])$ starts with $j$, which yields 
$$ \beta-k = \dfrac{j}{\beta} + \dfrac{1}{\beta}S_\beta(\beta-[\beta]). $$
Since $0\le S_\beta(\beta-[\beta])<1$, we get the announced result.
\end{proof}

\begin{proof}[Proof of Proposition~\ref{Prop:staircase}]
 The main tool here is the geometrical interpretation of the transformation $T_\lambda$ developped in Sections~\ref{Section:Geometrical Interpretation of P_i} and~\ref{Section:Geometrical Interpretation}.

Let us fix $0<\varepsilon<1/2$, $0<r<1/2-\varepsilon$.

Let $\tau\egdef k+r$ for some large integer $k$. Let $\theta\egdef\pi/\tau$. We want to find the beginning of $\omega_\lambda(\infty)=\infty_0 \infty_1\ldots$ for $\lambda\egdef 2\cos\theta$.
By Proposition~\ref{Prop:i_lambda}, we already know that $\infty_0=i_\lambda=k-1$. We introduce the points $M_j$, $0\le j\le k+1$ as defined on Figure~\ref{Fig:cercle}: These points lie on a circle centered at the origin, $M_0$ and $M_1$ have respective abscissae $0$ and $1$, and $M_{j+1}$ is the image of $M_j$ by the rotation of angle $\theta$. Denote by $t_j$ the abscissa of $M_j$: We have $t_k=R\sin k\theta>0$ and $t_{k+1}=R\sin (k+1)\theta<0$, where $R$ is the radius of the circle. 
Note that 
\begin{equation}
\label{eq:abscisses}
 \dfrac{-t_{k+1}-t_k}{t_k}\mathop\sim_{k\to\infty} \dfrac{1-2r}{r}
\end{equation}
To estimate geometrically $\infty_1$, we have to introduce a new circle centered at the origin, and two points $N_0$ and $N_1$ lying on this circle so that their abscissae are respectively $-t_{k+1}$ and $t_k$, and the angle $(ON_0,ON_1)$ equals~$\theta$. We again define the sequence $(N_j)$ of points on this circle by successive rotations of angle $\theta$. Then $\infty_1$ is the smallest $j$ such that $N_{j+2}$ has negative abscissa. Let $\varphi$ be the argument of $N_1$, so that the argument of $N_0$ is $\varphi-\theta$. We have
\begin{equation*}
 \dfrac{-t_{k+1}-t_k}{t_k} = \cos\theta - 1 + \sin\theta\tan\varphi.
\end{equation*}
Hence, by~\eqref{eq:abscisses}, observing that $\theta\tend{k}{\infty}0$, we get
$$ \lim_{k\to\infty} \sin\theta\tan\varphi = \dfrac{1-2r}{r} >4\varepsilon.$$
When $k\to\infty$, $\varphi\to\pi/2$ uniformly with respect to $r$. This implies that $\infty_1/k\to 0$ uniformly with respect to $r$. By Lemma~\ref{Lemma:asymptoticOfBeta}, we conclude that 
$$ \left|\beta\left(2\cos\left(\frac{\pi}{\tau}\right)\right) - (k-1)\right| \tend{k}{\infty} 0, $$
uniformly with respect to $r\in]0,1/2-\varepsilon[$.

It remains to study the case where $\tau\egdef k-r$. We now have $k-1<\tau<k$, and  Proposition~\ref{Prop:i_lambda} gives here $\infty_0=k-2$. The new points $N_0$ and $N_1$ have respective abscissae $-t_k=-R\sin k\theta$ and $t_{k-1}=R\sin(k-1)\theta$.
The analog of the estimation~\eqref{eq:abscisses} is
$$ \dfrac{t_{k-1}+t_k}{t_{k-1}} \mathop\sim_{k\to\infty} \dfrac{1-2r}{1-r}, $$
and the lefthand side is equal to $1 - \cos\theta - \sin\theta\tan\varphi$ (where $\varphi$ is the argument of $N_1$). We get now
$\varphi\tend{k}{\infty} -\pi/2$  uniformly with respect to $r$. This implies that $\infty_1/k\to 1$ uniformly with respect to $r$. We conclude by using Lemma~\ref{Lemma:asymptoticOfBeta}.
\end{proof}

\begin{corollary}
\label{Cor:non-analytic}
 The map $\lambda\mapsto\beta(\lambda)$ is not analytic.
\end{corollary}

\begin{proof}
 We know that $\beta(2\cos(\pi/k))=k-1$ for all integer $k\ge3$. If the map $\lambda\mapsto\beta(\lambda)$ were analytic, we would then have $\beta(2\cos(\pi/\tau))=\tau-1$ for all real $\tau>2$. This is clearly not the case by Proposition~\ref{Prop:staircase}.
\end{proof}

\section{Open questions}

\subsection{Invariant measures}
It can be shown that in the case $\lambda=\lambda_k$, the transformation $T_\lambda$ admits a unique absolutely continuous invariant measure $\mu$, whose density with respect to the Lebesgue measure is $\frac{d\mu}{dx}(x)=1/x$. Can we describe absolutely continuous $T_\lambda$-invariant measures for other $\lambda$'s?

\subsection{Algebraic properties of $\lambda$}
Many works have been devoted to the connections between algebraic properties of $\beta$ and the dynamical properties of the associated $\beta$-shift (see e.g. \cite{blanchard1989}). Are these properties also connected to the algebraic properties of the corresponding $\lambda$? We can stress the fact that the particular values of $\lambda$ studied in Section~\ref{Sec:map} are algebraic, and are always associated to an algebraic $\beta$. Does the correspondence between $\lambda$ and $\beta$ always associate algebraic numbers to algebraic numbers?

\subsection{Lazy and random expansions in $\lambda$-continued fractions}
For $\beta>1$, there are generally many ways to expand a real number $t\in[0,1[$ in the form 
$$ t = \sum_{n\ge0}\frac{t_n}{\beta^{n+1}}, $$
where the $t_n$'s are taken from $\{0,1,\ldots,\lfloor \beta \rfloor\}$. The expansion given by iteration of the map $S_\beta$ defined in Section~\ref{Sec:beta-shifts} is sometimes refered to as the \emph{greedy expansion}, since at each step it uses the largest possible digit.
It therefore gives the maximal expansion of $t$ with respect to the lexicographic order. Symmetrically, one can consider the minimal $\beta$-expansion of $t$, given by iteration of the so-called \emph{lazy map} which at each step outputs the smallest possible digit (see~\cite{dajani2005} and references therein). Between the lazy and the greedy expansions, one can find a continuum of possible $\beta$-expansions among which the \emph{random $\beta$-expansions} studied in~\cite{dajani2005,dajani2007}. Can we develop some expansions in $\lambda$-continued fractions corresponding to these notions of lazy and random $\beta$-expansions?

\subsection{Regularity of $\lambda\mapsto\beta(\lambda)$}
We proved in Section~\ref{Sec:map} that the map  $\lambda\mapsto\beta(\lambda)$ is increasing and continuous but non-analytic. 
A careful look at the details in Figure~\ref{Fig:beta(t)} suggests some kind of self-similarity in the graph. 
We thus may expect the map to be not even differentiable.

\bibliography{lambda_beta.bib}

\end{document}